\documentclass[reqno]{amsart}

\numberwithin{equation}{section} 
\usepackage{amsmath,amsfonts,amssymb,amsthm,latexsym}
\usepackage{enumerate}
\usepackage{cancel}
\usepackage{mathrsfs}
\usepackage{stackrel}

\usepackage{cases}
\usepackage[square,comma,numbers,sort&compress]{natbib}

\usepackage[colorlinks=true]{hyperref}
\hypersetup{urlcolor=blue, citecolor=red}

\newcounter{mnote}
\theoremstyle{plain}
\newtheorem{theorem}{Theorem}[section]

\newtheorem{lemma}[theorem]{Lemma}

\theoremstyle{definition}
\newtheorem{definition}[theorem]{Definition}

\theoremstyle{remark}
\newtheorem{remark}[theorem]{Remark}

\newcommand{\field}[1]{\mathbb{#1}}

\newcommand{\nR}{\field{R}}

\newcommand{\nS}{\field{S}}

\newcommand{\veps}{\varepsilon}

\newcommand{\abs}[1]{\left\lvert#1\right\rvert}
\newcommand{\norm}[1]{\left\lVert#1\right\rVert}
\newcommand{\Lp}[1]{#1}

\newcommand{\B}[3]{\text{${B}^{#1}_{#2, #3}$}}

\begin{document}
\title[Infinity theory]{The space $B^{-1}_{\infty, \infty}$, volumetric sparseness, and 3D NSE}

\date{\today}

%
\author{Aseel Farhat}
\address[Aseel Farhat]{Department of Mathematics\\
               University of Virginia\\
       Charlottesville, VA 22904, USA}
\email[Aseel Farhat]{af7py@virginia.edu} 

\author{Zoran Gruji\'c}
\address[Zoran Gruji\'c]{Department of Mathematics\\
               University of Virginia\\
       Charlottesville, VA 22904, USA}
\email[Zoran Grujic]{zg7c@virginia.edu} 

\author{Keith Leitmeyer}
\address[Keith Leitmeyer]{Department of Mathematics\\
               University of Virginia\\
       Charlottesville, VA 22904, USA}
\email[Keith Leitmeyer]{kl2ju@virginia.edu}  
\begin{abstract}

In the context of the $L^\infty$-theory of the 3D NSE, it is shown that
smallness of a solution in Besov space $B^{-1}_{\infty, \infty}$
suffices to prevent a possible blow-up. In particular, it is revealed that
the aforementioned condition implies a particular
local spatial structure of the regions of intense velocity components, namely,
the structure of local volumetric sparseness on the scale comparable
to the radius of spatial analyticity measured in $L^\infty$.

\end{abstract}

 \maketitle

\section{Introduction}\label{intro}

\noindent  Motion of 3D incompressible, viscous fluid is modeled by 3D Navier-Stokes
equations (NSE),
\[
 u_t+(u\cdot \nabla)u=-\nabla p + \triangle u,
\]
supplemented with the incompressibility condition $ \, \mbox{div} \,
u = 0$, where $u$ is the velocity of the fluid and $p$ is the
pressure (here, the viscosity is set to $1$, and the external force to $0$).
Henceforth, the spatial domain will be the whole space $\mathbb{R}^3$.

\medskip

A question of whether the 3D NSE allow a formation of singularities is
an open problem; moreover, the problem is \emph{super-critical} in the sense that there is
a fixed `scaling distance' between any presently known regularity criterion and the corresponding
(presently known)
\emph{a priori} bound. A telling example is a highly nontrivial regularity criterion obtained in
\cite{ESS03}, namely, $u \in L^\infty (0,T; L^3)$, to be contrasted to \emph{a priori} boundedness
of the kinetic energy, $u \in L^\infty (0,T; L^2)$, satisfied by any Leray weak solution.
In particular, the regularity criteria are (at best) scaling-invariant with respect to the
unique scaling leaving the equations invariant.

\medskip

A (very) partial hierarchy of the scaling invariant spaces of interest $X$ is as follows,

\[
 L^3 \hookrightarrow L^{3, \infty} \hookrightarrow BMO^{-1} \hookrightarrow B^{-1}_{\infty, \infty}.
\]

\medskip

Looking at the corresponding (existing) regularity criteria in $L^\infty (0,T; X)$, the only criterion
that does not require a smallness condition is the aforementioned result of Escauriaza,
Seregin and Sverak \cite{ESS03} in $L^\infty (0,T; L^3)$. On the other side of the spectrum, since 
$B^{-1}_{\infty, \infty}$ 
is the largest scaling-invariant space in play, obtaining even a smallness regularity criterion in
$L^\infty (0,T; B^{-1}_{\infty, \infty})$ is of significant interest.
A positive answer was given in
\cite{CP02} and \cite{ChSh10}, in the setting of mild solutions belonging to a suitable
$p$-integrable ($p < \infty$) Besov space and
Leray solutions, respectively. 

\medskip

Since $B^{-1}_{\infty, \infty}$ is an
$\infty$-type space, a natural question becomes whether it is possible to derive a smallness 
regularity criterion in
$L^\infty (0,T; B^{-1}_{\infty, \infty})$
without assuming \emph{any global integrability} of solutions, i.e., 
in the context of the $L^\infty$-theory. A class of weak/distributional solutions where this is relevant
is the class of uniformly-local, non-decaying `local Leray solutions' constructed
by Lemari\'e-Rieusset (cf. \cite{LR02}), where a spatial singular set 
at a (possible) singular time $T$ could consists of a sequence of points $\{x_j\}$ running off to infinity. 
More specifically, we could envision a criticality scenario where the singularity build up at each $(x_j, T)$
would feature a locally self-similar blow-up rate (such a solution could be a local Leray solution on a 
time-interval containing $T$). If, in addition, we suppose that $T$ is the first singular time, then the 
solution--up to $T$--could be in $L^\infty$, but not in \emph{any} proper $L^p$.

\medskip

It is worth noting that, compared to the $L^p$-theory,
the $L^\infty$-theory of the 3D NSE, as well as of the Stokes problem, is less established.
A good illustration of this fact is that even such a fundamental question as whether the Stokes semigroup
generates an analytic semigroup in an $L^\infty$-type space on domains with boundaries was 
addressed only recently
(see \cite{AG13} for the case of a bounded domain with no-slip boundary conditions, and
\cite{AG14} for the case of an exterior domain).

\medskip

In this short note we give an affirmative answer to the above question; more precisely,
we prove the following theorem.

\begin{theorem}\label{infinity}
Let $u$ be a unique mild solution to the 3D NSE emanating from an initial
datum $u_0$ in $L^\infty$, and $T>0$ be the first possible blow-up time. 
There exists a positive (absolute) constant $m_0$ such that if the solution $u$
satisfies
\[ 
 \sup_{t \in (T-\epsilon, T)} \|u(t)\|_{B^{-1}_{\infty, \infty}} \le m_0,
\]
for some $0 < \epsilon <T$, then $T$ is not a blow-up time, and the solution
can be continued past $T$.
\end{theorem}

The idea of the proof is to utilize a recent $L^\infty$-theory of formulating geometric
regularity criteria for the 3D NSE based on \emph{local 1D sparseness} of the super-level sets 
\cite{Gr13}
in conjunction
with a technical lemma quantifying the amount of 3D/volumetric sparseness of a super-level set
imposed by
membership of the vector field in view in the space $B^{-1}_{\infty, \infty}$.
It is worth mentioning that the argument reveals that the assumption on local sparseness
of the super-level sets is in fact a \emph{weaker condition} than the assumption on the smallness
of the solution in $L^\infty (0,T; B^{-1}_{\infty, \infty})$ (see Remark \ref{yay}).

\medskip

A related avenue to understanding the role that scaling-invariant spaces play in
the regularity theory of the 3D NSE is consideration of well-posedness for 
small initial data.
The best result in this direction so far is the result of Koch and Tataru \cite{KT01},
taking the initial data in $BMO^{-1}$. The question of whether a
small initial data result is possible in the largest scaling-invariant space $B^{-1}_{\infty, \infty}$
is still open; however, there are indications that the answer to this question might be negative
(see, e.g., \cite{BP08}).

\medskip

The note is organized as follows. Section 2 contains the preliminaries regarding 
the role that local 1D sparseness of the super-level sets plays in controlling
the $L^\infty$-norm of the solution, Section 3 presents a technical lemma connecting
the space $B^{-1}_{\infty, \infty}$ to  3D sparseness of the super-level sets,
and Section 4 contains the proof of the above theorem, and a remark on a scenario in 
which the smallness condition is not needed.

\section{Sparseness} 

\noindent The concept of `local 1D sparseness' of a set has recently emerged
in the study of geometric conditions preventing possible formation of
singularities in the 3D NSE (cf. \cite{Gr13}).

\medskip

Let
$S \subseteq \nR^3$ be an open set, $x_0$ a point in $\nR^3$,
$r>0$, and $\delta \in (0,1)$ ($m$ will denote the Lebesgue measure).

\begin{definition}\label{1D_sparse} 
$S$ is {\it 1D $\delta$-sparse around $x_0$ at scale $r$} if there exists a unit vector 
${\bf d}$ in $\nS^2$ such that 
$$ 
\frac{m(S\cap (x_0-r{\bf d}, x_0+r{\bf d}))}{2r} \leq \delta.
$$
\end{definition} 

The volumetric version is as follows.

\begin{definition}\label{3D_sparse} 
$S$ is {\it 3D $\delta$-sparse
around $x_0$ at scale $r$} if 
$$
\frac{m(S\cap B(x_0,r))}{m(B(x_0,r))} \leq \delta.
$$
\end{definition}

\begin{remark}
It is plain that if $S$ is 3D $\delta$-sparse around $x_0$ at scale $r$,
then $S$ is automatically 1D $(\delta)^{1/3}$-sparse around $x_0$
at scale $\rho$, for some $0 < \rho \le r$.
(This is easily seen by assuming the opposite; then, integrating the
characteristic function of $S \cap B(x_0, \rho)$ in `polar' 
coordinates--assuming the worst case scenario--yields the 
contradiction.)
\end{remark}

The main idea of how the local sparseness of the super-level sets is used
in conjunction with the spatial analyticity of solutions to obtain 
control of the $L^\infty$-norm is very simple (the super-level sets considered
here are the regions in which the magnitude is above a fraction of the
$L^\infty$-norm). Intuitively, a high degree of sparseness near a possible
blow-up time indicates a high level of \emph{spatial complexity} (e.g., rapid spatial 
oscillations) that eventually becomes incompatible with the uniform 
local-in-time spatial analyticity 
properties of solutions, leading to a contradiction (as in a typical blow-up
argument). Technically, this is realized via the harmonic measure
maximum principle (see \cite{Gr13}).

\medskip

One should mention that the morphology of the regions of \emph{intense
velocity} and the regions of \emph{intense vorticity} in turbulent flows is 
quite different. On one hand, the velocity regions are (in the average)
homogeneous and \emph{isotropic}, while on the other hand, the vorticity regions 
are (in the average) \emph{locally anisotropic} and dominated by
\emph{vortex filaments} \cite{S81, SJO91, JWSR93, VM94, CPS95}. 
However, in both cases, a geometric signature is the one of
sparseness; more precisely, 3D sparseness and (local) 1D sparseness,
respectively. A mathematical story about how the interplay between vortex stretching and 
locally anisotropic diffusion might lead to closing the `scaling gap' in
the 3D NS regularity problem--motivated by G.I. Taylor's view on turbulent
dissipation \cite {Tay37}--was presented in \cite{Gr13, DaGr12-3, BrGr13-2}; here, we address
the scenario of the volumetric (3D) sparseness of the regions of intense
velocity.

\medskip

Since the local-in-time spatial analyticity properties of solutions play
a key role in the theory, we recall a variant of the pertinent result obtained in \cite{Gu10},
inspired by the method of finding a lower bound on the uniform radius of
spatial analyticity of solutions in $L^p$ spaces introduced in 
\cite{GrKu98}.

\medskip

\begin{theorem}\emph{[Gu10]}\label{an_u}
Let $u_0$ be in $L^\infty$. Then, for any $M>1$, there exist
$C(M)$ and $\widetilde{C}(M)$, 
such that setting
$\displaystyle{T=\frac{1}{{C(M)}^2 \|u_0\|_\infty^2}}$, a unique mild
solution $u=u(t)$ on $[0,T]$ has the analytic extension $U=U(t)$ to
the region
\[
 \mathcal{R}_t = \{x+iy \in \mathbb{C}^3 : \, |y| \le \frac{1}{\widetilde{C}(M)}
 \sqrt{t}\}
\]
for any $t$ in $(0,T]$, and
\[
 \|U(t)\|_{L^\infty(\mathcal{R}_t)} \le M \|u_0\|_\infty
\]
for all $t$ in $[0,T]$.
\end{theorem}

\medskip

(For the results on spatial analyticity of the 3D NSE in the critical
Besov spaces, see \cite{BBT12}.)

\medskip

Then, a variant of the main result in \cite{Gr13} reads as follows.

\medskip

\begin{theorem}\emph{[Gr13]}\label{sparse_u}
Suppose that a solution $u$ is regular on an interval $(0,T^*)$.
(Recall that $u$ is then necessarily in $C\Bigl((0,T^*);
L^\infty\Bigr)$.)

\medskip

Let $M$ be the solution to the equation $\frac{1}{2}h+(1-h)M=1$, where
$h=\frac{2}{\pi}\arcsin\frac{1-(\frac{3}{4})^\frac{2}{3}}{1+(\frac{3}{4})^\frac{2}{3}}$,
and let  $C(M), \widetilde{C}(M)$ be as in Theorem \ref{an_u}
(note that $M>1$).
Assume that there exists
$\epsilon >0$ such that for any $t$ in $(T^*-\epsilon, T^*)$, either

\medskip

(i) \ $\displaystyle{t+\frac{1}{C(M)^2 \|u(t)\|_\infty^2} \ge T^*}$, or

\medskip

(ii) \ there
exists $s=s(t)$ in $\Bigl[t+\frac{1}{4C(M)^2 \|u(t)\|_\infty^2},
t+\frac{1}{C(M)^2 \|u(t)\|_\infty^2}\Bigr]$ such that for any spatial point
$x_0$, there exists a scale $r$, $0<r\le \frac{1}{2 C(M) \widetilde{C}(M)
\|u(t)\|_\infty}$, with the property that the component super-level sets
\[
 B^{i, \pm}_s=\{x \in \mathbb{R}^3: \, u_i^\pm(x,s) > \frac{1}{2}\|u(t)\|_\infty\}
\]
are 1D $(\frac{3}{4})^\frac{1}{3}$-sparse around $x_0$ at scale $r$, for $i=1,2,3$
(here, as customary, for a real-valued function $g$, $g^+(x)=\max\,(g(x),0)$ and
$g^-(x)=-\min\,(g(x),0)$).

\medskip

Then, there exists $\gamma >0$ such that $u$ is in
$L^\infty\Bigl((T^*-\epsilon, T^*+\gamma); L^\infty\Bigr)$, i.e.,
$T^*$ is not a singular time.
\end{theorem}

\medskip

This is a refinement of the theorem in \cite{Gr13} in the sense that instead of postulating
the sparseness of the full vectorial super-level set, only the sparseness of each of the
six component super-level sets is required. The reason that the proof remains
the same is that the argument is \emph{completely local}, i.e., we are estimating 
$|u(x_0,s_0)|$ one spatial point at a time, and since (considering the maximum
vector norm in $\mathbb{R}^3$) $|u(x_0,s_0)|$ is equal to one of the six
$u_i^\pm(x_0,s_0)$, we can simply apply the harmonic measure maximum 
principle to the \emph{subharmonic} function $u_i^\pm$.

\medskip

Note that in the statement of the above theorem (as well as in the original theorem), 
the super-level sets are considered
at a time $s(t)$, with respect to the level depending on a preceding time $t$. This is
not an optimal setting for the argument that we wish to make in the final section. To alleviate
this, we state a different version of the theorem in which \emph{everything} is
evaluated at the \emph{same} point in time; the trade-off is that this is possible only for
suitably chosen times based off the concept of an `escape time'. 

\begin{definition}\label{escape_time} 
Let $u_0$ be in $L^\infty$, $u$ a unique mild solution emanating form $u_0$,
and $T>0$ the first possible blow-up time. A time $t$ in $(0, T)$ is an
\emph{escape time} if $\|u(t)\|_\infty < \|u(\tau)\|_\infty$ for any $\tau$
in $(t, T)$.
\end{definition}

\medskip

\begin{remark}
Local-in-time well-posedness in $L^\infty$ implies that there are
continuum many escape times.
\end{remark}

Observing that
for an escape time $t$ and any $s(t)$ in $\Bigl[t+\frac{1}{4C(M)^2 \|u(t)\|_\infty^2},
t+\frac{1}{C(M)^2 \|u(t)\|_\infty^2}\Bigr]$,

\[
 \frac{1}{M} \|u(s(t))\|_\infty \le \|u(t)\|_\infty < \|u(s(t))\|_\infty,
\]

\medskip

\noindent a slight modification of the proof of the theorem yields the desired
version. More precisely, the utility of $t$ being an escape time is twofold.
Firstly, since the $L^\infty$-norms at $t$ and $s(t)$ are now comparable,
there is no need for the time-lag when setting the super-level set cut-off.
Secondly, in the original argument (\cite{Gr13}), the temporal points at which
the super-level sets were considered were organized in a finite sequence, and
the sparseness--via the harmonic measure maximum principle--guaranteed
that the distance between the consecutive points did not shrink, causing the sequence to
eventually surpass the possible singular time $T$, yielding the contradiction. In the
current setting, the contradiction is obtained in a single temporal
step as the sparseness-induced control on the $L^\infty$-norm contradicts the
defining property of being an escape time. 

\medskip

\begin{theorem}\label{escape_not}
Let $u_0$ be in $L^\infty$, $u$ a unique mild solution emanating form $u_0$,
$T>0$ the first possible blow-up time, and $t$ an escape time.

Suppose that there exists 
$s=s(t)$ in $\Bigl[t+\frac{1}{4C(M)^2 \|u(t)\|_\infty^2},
t+\frac{1}{C(M)^2 \|u(t)\|_\infty^2}\Bigr]$ such that for any spatial point
$x_0$, there exists a scale $\rho$, $0<\rho\le \frac{1}{2 C(M) \widetilde{C}(M)
\|u(s(t))\|_\infty}$, with the property that the component super-level sets
\[
 A^{i, \pm}_s=\{x \in \mathbb{R}^3: \, u_i^\pm(x,s(t)) > \frac{1}{2}\|u(s(t))\|_\infty\}
\]
are 1D $(\frac{3}{4})^\frac{1}{3}$-sparse around $x_0$ at scale $\rho$, for $i=1,2,3$.

\medskip

Then $T$ is not a blow-up time.
\end{theorem}

\section{Mixing}

\noindent A closely related concepts of a set being $r$-mixed (or `mixed to scale $r$')
and $r$-semi-mixed appeared in the study of 
rearrangements and
mixing properties of incompressible
flows (see, e.g., \cite{Br03} and \cite{IKX14}). 

\begin{definition}\label{semi_mixed} 
Let $r>0$.  An open set $S$ is $r$-\emph{semi-mixed} if
$$
\frac{m(S\cap B(x,r))}{m(B(x,r))} \leq \delta
$$ 
for every $x\in \nR^3$, and for some $\delta \in (0,1)$.
If the complement, $S^c$, is $r$-semi-mixed as well, then $S$ is said to be 
$r$-\emph{mixed}.
\end{definition}

\begin{remark}
If the set $S$ is $r$-semi-mixed (with the ratio $\delta$), then it is 3D $\delta$-sparse around every 
point $x_0 \in \nR^3$ at scale $r$.
\end{remark}

\noindent The following lemma is a vector-valued, Besov space version of a scalar-valued, Sobolev space
lemma in \cite{IKX14}. All the norms to appear in the statement of
the lemma are $\infty$-type norms.

\bigskip

\begin{lemma}\label{mixing_lemma}
Let $\epsilon \in (0,1], \, r \in (0,1]$ and $u$ a vector-valued function in $L^\infty$.
Then, for any pair $\lambda, \delta$, $\lambda \in (0,1)$ and $\delta \in (\frac{1}{1+\lambda}, 1)$, 
there exists an explicit constant 
$c=c(\lambda, \delta)$ such that if 
$$\norm{u}_{\B{-\veps}{\infty}{\infty}} \leq c(\lambda, \delta) \, r^{\veps} \norm{u}_{\Lp{\infty}}, $$
then each of the six super-level sets $A_\lambda^{i, \pm} := \left\{ x\in \nR^3: u_i^{\pm} > \lambda \norm{u}_{\Lp{\infty}}\right\}$ is 
$r$-semi-mixed with the ratio $\delta$. 
\end{lemma} 

\begin{proof}

Arguing by contradiction, assume that there exists $i \in\{1,2,3\}$ such that either 
$A^{i,+}_\lambda$ or $A^{i,-}_\lambda$ is not $r$-semi-mixed with the ratio $\delta$.
Without loss of generality, assume that it is $A^{i,+}_\lambda$ (if it were $A^{i,-}_\lambda$, the
only modification to the proof would be replacing the function $f$ below with $-f$).

\medskip

Then,
there exists $x_0\in \nR^3$ such that
$$ \frac{m(A^{i,+}_\lambda \cap B(x_0,r))}{m(B(x_0,r))} > \delta;$$ equivalently,
$$m(A^{i,+}_\lambda \cap B(x_0,r)) > \delta \, \Pi(3) \, r^3, $$
where $\Pi(3)$ denotes the volume of the unit ball in $\nR^3$. 

\medskip

Let $f \in \B{\veps}{1}{1}$ be a smooth radial cut-off equal to 1 in $B(x_0,r)$, and vanishing
outside $B(x_0, (1+\eta)r)$ for some $\eta>0$ (the value to be determined at the end of the proof). Then, 
\begin{equation}\label{norm_u}
\norm{u}_{\B{-\veps}{\infty}{\infty}} \geq \frac{c}{\norm{f}_{\B{\veps}{1}{1}}} \abs{\int_{\nR^3} u_i(x)f(x)\, dx}
\end{equation}
for a positive constant $c$.

\medskip

An explicit calculation of the $\B{\veps}{1}{1}$-norm of $f$ via the finite differences
(using the finite differences of order two for the endpoint case $\veps = 1$; c.f. chapter II in \cite{BCD11}) yields

\begin{equation}\label{norm_f}
\norm{f}_{\B{\veps}{1}{1}}  \leq c(\eta) \, r^{3-\veps}
\end{equation}
for some $c(\eta)>0$.

\medskip

Next, write
\begin{align*}
\abs{\int_{\nR^3} u_i(x)f(x)\, dx} \geq \int_{\nR^3} u_i(x)f(x)\, dx
\geq I - \abs{II} - \abs{III}, 
\end{align*}
where $$I = \int_{A^{i,+}_\lambda \cap B(x_0,r)} u_i(x)f(x)\, dx,$$ $$II = \int_{B(x_0,r)\backslash A^{i,+}_\lambda}u_i(x)f(x)\, dx,$$ and $$III = \int_{(B(x_0,(1+\eta)r)\backslash B(x_0,r))} u_i(x)f(x)\, dx.$$ 

\medskip

It is plain that 
\begin{align}\label{I}  
I &= \int_{A^{i,+}_\lambda \cap B(x_0,r)} u_i(x)\,dx =  \notag \int_{A^{i,+}_\lambda \cap B(x_0,r)} u_i^+(x)\,dx
\\ &  > \lambda \norm{u}_{\Lp{\infty}}m(A^{i,+}_\lambda \cap B(x_0, r)) \geq  \lambda \, \delta \, \Pi(3) \, r^3 \norm{u}_{\Lp{\infty}},
\end{align}
\begin{align}\label{II}
\abs{II} & = \abs{\int_{B(x_0,r) \backslash A^{i,+}_\lambda}u_i(x)\, dx} \leq \norm{u}_{\Lp{\infty}} \left( m(B(x_0,r) - m(A^{i,+}_\lambda \cap B(x_0,r))\right)\notag \\
& \leq \norm{u}_{\Lp{\infty}} \left(\Pi(3) r^3 - \delta \Pi(3) r^3\right) \notag \\
& = \left(1 - \delta \right)\Pi(3) \, r^3 \norm{u}_{\Lp{\infty}}, 
\end{align}
and 
\begin{align}\label{III}
\abs{III} & \leq \abs{\int_{(B(x_0,(1+\eta)r)\backslash B(x_0,r))} u_i(x)\, dx}\notag \\
& \leq \norm{u}_{\Lp{\infty}} \left( m(B(x_0,(1+\eta)r) - m(B(x_0,r))\right)\notag \\
& \leq \left((1+\eta)^3 - 1\right)\Pi(3)r^3 \norm{u}_{\Lp{\infty}}. 
\end{align}

\medskip

It follows from \eqref{norm_u}, \eqref{norm_f} and \eqref{I}--\eqref{III} that 
\begin{align*}
\norm{u}_{\B{-\veps}{\infty}{\infty}} > c^*(\eta) \, \Pi(3) \, r^{\veps} \norm{u}_{\Lp{\infty}} 
(\lambda \delta + \delta  - (1+\eta)^3). 
\end{align*}

\medskip

Since $\delta > \frac{1}{1+\lambda}$, we can define $\eta=\eta(\lambda, \delta)$ to be the solution of the equation
$(1+\eta)^3 = \frac{\delta (1+\lambda) +1}{2}$; this in turn yields 
 
\begin{align}
\norm{u}_{\B{-\veps}{\infty}{\infty}} > c(\lambda, \delta)\Pi(3)r^{\veps} \norm{u}_{\Lp{\infty}}
\end{align}
with $c(\lambda, \delta) = c^*(\eta) \frac{\delta (1+\lambda) -1}{2}$, which is positive since
$\delta > \frac{1}{1+\lambda}$. This contradicts the statement in the lemma. 
\end{proof}

\medskip 

It has already been observed--in the context of the Sobolev $H^{-k}$-spaces--that the converse of 
this type of result is not necessarily true (see  \cite{IKX14}). Here, we present a simple
counterexample to the converse of the above lemma in the case $\epsilon=1$.
The function $f$ will be a 
`dome with a lightning rod' constructed as follows: let $g=g(r)$ be a function on $[0, \infty)$
obtained by smoothing out the edges of the polygonal line connecting the points
$(0, 2), (1/n, 1), (1, 1), (2,0)$ and $(\infty, 0)$, and set $f(x)=g(|x|)$. (This can be done with
the optimal bounds on the slopes of the secant lines to the graph of $g$ analogous to the 
optimal bounds on the slopes of the tangent lines/derivatives when constructing standard
cut-off functions.) On one hand, a simple geometric argument implies that
the full vectorial super level set $\{|f| > \frac{3}{4} \|f\|_\infty\}$ is $\frac{2}{n}$-mixed
with the ratio $\delta = \frac{1}{8}$. On the other hand, 
the inequality
\[
\|f\|_{ B^{-1}_{\infty, \infty} } \leq  c(\lambda, \delta) \, r \,\|f\|_\infty
\]
is doomed.
More precisely, since the scale $r$ of interest is now $r=\frac{2}{n}$, it is plain that
the term $r \,\|f\|_\infty$ is equal to $\frac{4}{n}$, while the computation of 
$\|f\|_{ B^{-1}_{\infty, \infty} }$ via `duality' as in (\ref{norm_u}) (testing $f$ against
itself) yields at least
$O(1)$; namely, both the $L^2$-norm and the $B^1_{1, 1}$-norm
of $f$ are $O(1)$, the latter following from a calculation via the finite differences.

\medskip

\begin{remark}\label{yay}
In the context of the study of the regularity theory of the 3D NSE, the above
example is interesting as it indicates that the assumption on local sparseness of
the super-level sets (Theorem \ref{escape_not}) is a weaker condition than
the smallness assumption in the Besov norm $B^{-1}_{\infty, \infty}$ (c.f. the proof of the 
main result below).
\end{remark}

\section{Proof of the theorem and some thoughts on non-smallness}

\noindent In this section, we present a short proof of the main result, and a discussion
about a blow-up scenario in which the smallness condition is not needed.

\begin{proof}

Let $t$ be an escape time in $(T-\epsilon, T)$, and $s(t)$ a time in the interval
$\Bigl[t+\frac{1}{4C(M)^2 \|u(t)\|_\infty^2}, t+\frac{1}{C(M)^2 \|u(t)\|_\infty^2}\Bigr]
\subset (T-\epsilon, T)$. Consider any of the component super-level sets

\[
 A^{i, \pm}_s=\{x \in \mathbb{R}^3: \, u_i^\pm(x,s(t)) > \frac{1}{2}\|u(s(t))\|_\infty\}.
\]

By Lemma \ref{mixing_lemma}, setting $\epsilon$ to be $1$, $\lambda= \frac{1}{2}$, 
and $\delta=\frac{3}{4}$,
there exists a constant $c_*>0$ such that the condition
\[
 \norm{u(s(t))}_{\B{-1}{\infty}{\infty}} \leq c_* \, r \norm{u(s(t))}_{\Lp{\infty}}
\]
implies that the super-level set $A^{i, \pm}_s$ is $r$-semi-mixed; in particular, this holds for
$r=\frac{1}{2 C(M) \widetilde{C}(M) \|u(s(t))\|_\infty}$. Consequently, the condition
\[
  \norm{u(s(t))}_{\B{-1}{\infty}{\infty}} \leq \frac{c_*}{2 C(M) \widetilde{C}(M)}
\]
implies that $A^{i, \pm}_s$ is $\frac{1}{2 C(M) \widetilde{C}(M) \|u(s(t))\|_\infty}$-semi-mixed with the ratio
$\delta=\frac{3}{4}$, which in turn,
implies that $A^{i, \pm}_s$ is 1D $(\frac{3}{4})^\frac{1}{3}$-sparse at some scale $\rho$, 
$0 < \rho \le \frac{1}{2 C(M) \widetilde{C}(M) \|u(s(t))\|_\infty}$, for any $x_0$ in $\mathbb{R}^3$.
Setting $m_0 = \frac{c_*}{2 C(M) \widetilde{C}(M)}$, all the conditions in Theorem \ref{escape_not}
are satisfied, and $T$ is not a blow-up time.
\end{proof}

\begin{remark}
Note that in the statement of the theorem, the smallness condition is imposed over some
interval $(T-\epsilon, T)$; however, the proof reveals that the condition is needed only
at a \emph{single time} $s(t)$.
\end{remark}

\medskip

At the end of this note, we would like to offer a possible scenario in which the smallness condition
is not needed. Arguing the same way as in the proof of the theorem--but utilizing 
Lemma \ref{mixing_lemma}
with an $\epsilon$ in $(0,1)$ instead--we see that the condition assuring the application
of Theorem \ref{escape_not} can be formulated as

\medskip

\[
 \norm{u(s(t))}_{\B{-\epsilon}{\infty}{\infty}} \leq  \frac{c_*}{2 C(M) \widetilde{C}(M)}  
 \norm{u(s(t))}_{\Lp{\infty}}^{1-\epsilon}.
\]
Since the optimization in $\epsilon$ will not lead to a qualitatively different result,
we set $\epsilon = \frac{1}{2}$, and rewrite the above as

\medskip

\[
 \norm{u(s(t))}_{\B{-\frac{1}{2}}{\infty}{\infty}} 
 \leq  \frac{c_*}{2 C(M) \widetilde{C}(M)}  
 \biggl(\frac{\norm{u(s(t))}_{\Lp{\infty}}}{ \norm{u(s(t))}_{\B{0}{\infty}{\infty}}}\biggr)^\frac{1}{2}
 \norm{u(s(t))}_{\B{0}{\infty}{\infty}}^\frac{1}{2}.
\]
Unraveling the characterization of the Besov norms in terms of the Littlewood-Paley
decomposition yields that it is sufficient to require 

\medskip

\[
  \norm{u(s(t))}_{\B{-1}{\infty}{\infty}}
  \leq  \frac{c_*^2}{4 C(M)^2 \widetilde{C}(M)^2}  
  \frac{\norm{u(s(t))}_{\Lp{\infty}}}{ \norm{u(s(t))}_{\B{0}{\infty}{\infty}}}.
\]
This provides an (admittedly narrow) escape route from the smallness; namely,
for certain \emph{classes} of functions, the ratio 

\[
 \frac{\|f\|_{\infty}}{\|f\|_{B^{0}_{\infty,\infty}}}
\]
can become \emph{arbitrarily large}. A typical example is given 
by the mollifications of the logarithm. More precisely, for 
$\epsilon>0$, set

\[
f_{\epsilon} = \rho_{\epsilon}* (\log^+(1/|x|)),
\]

\noindent where $\rho_{\epsilon}$ is the standard mollifier. Then, on one hand,
$\|f_{\epsilon}\|_{\infty}=O(\log(1/\epsilon))$, while on the other hand,
$\|f_{\epsilon}\|_{B^{0}_{\infty,\infty}}=O(1)$,
as $\epsilon\to 0$. The first asymptotics is transparent, and to see that
the second one is at most $O(1)$ (which is what we need), recall that
the inclusion of $BMO$ in $B^0_{\infty, \infty}$ is continuous, and 
observe that 
$\|f_{\epsilon}\|_{BMO}=O(1)$; this follows from the scaling properties
of the convolution and the $BMO$-norm similarly to the standard
argument that $\|\log |x|\|_{BMO}$ is finite (cf. \cite{St93}, the first section
in chapter IV).
Consequently, the ratio

\[
 \frac{\|f_{\epsilon}\|_{\infty}}{\|f_{\epsilon}\|_{B^{0}_{\infty,\infty}}}
\]

\noindent can indeed become arbitrarily large. 

\medskip

Of course, for this to be relevant,
the question becomes whether it is
realistic to expect that the local spatial structure of the flow around a 
possible singular point $x^*$, at a time near a possible
blow-up time $T$, can exhibit a log-like profile. If we needed the condition to hold
for a sequence of times converging to $T$, the answer
would be negative, as the spatial profile at $(x^*, T)$
has to be (essentially) at least as singular as $\frac{1}{|x-x^*|}$
(see, e.g., \cite{Ko98} where it was shown that a spatial singularity
to the 3D NSE that is $o\biggl(\frac{1}{|x-x^*|}\biggr)$ is, in fact, removable).
However,
since it suffices that the condition holds at a single time $s(t)$ (out of
continuum many available), one could envision a scenario in which
there is a log-like transition to the algebraic singularity that would
include the spatial profile at $s(t)$.

\section*{Acknowledgements}

\noindent The work of A.F. is supported in part by the \emph{National Science Foundation} grant DMS-1418911.
The work of Z.G. is supported in part by
the \emph{Research Council of Norway} grant
213474/F20 and the \emph{National Science Foundation} grant
DMS 1515805. The authors would like to thank the referee for the constructive criticism.

\end{document}